\documentclass[12pt]{article}

\usepackage{amsmath}             
\usepackage{amssymb}             
\usepackage{amsfonts}            
\usepackage{mathrsfs}            
\usepackage{amsthm}
\usepackage{mathtools}
\usepackage{commath}
\usepackage{bm}
\usepackage{graphicx}
\usepackage{geometry}
\usepackage{float}
\usepackage{listings}
\usepackage{subfigure}
\usepackage{multirow}
\usepackage{diagbox}
\usepackage[export]{adjustbox}
\usepackage[all]{xy}
\usepackage{tikz-cd}
\usepackage{fancyhdr}
\usepackage{hyperref}
\usepackage{appendix}

\DeclareMathAlphabet{\mathbbb}{U}{bbold}{m}{n}

\theoremstyle{definition}
\newtheorem{defi}{Definition}[section]

\theoremstyle{plain}
\newtheorem{theo}[defi]{Theorem}
\newtheorem{question}{Question}[section]
\newtheorem{lemma}[defi]{Lemma}
\newtheorem{prop}[defi]{Proposition}
\newtheorem{cor}[defi]{Corollary}

\newtheorem*{conj}{Conjecture}

\theoremstyle{remark}
\newtheorem*{rmk}{Remark}

\begin{document}

\title{A Conjecture of Bhatt--Lurie and weakly $p$-nilpotent Hodge--Tate stacks}
\author{Jiahong Yu
\thanks{J.Y., yu\_jh@stu.pku.edu.cn, Beijing International Center for Mathematical Research, Peking University, YiHeYuan Road 5, Beijing, 100190, China.}}

\def\Z{\mathbb{Z}}
\def\N{\mathbb{N}}
\def\Q{\mathbb{Q}}
\def\Prism{\mathbbb{\Delta}}
\def\sq{\square}
\def\a{\mathfrak{a}}
\def\Ord{\mathbf{\Delta}}
\def\ep{\epsilon}
\def\C{\mathscr{C}}
\def\inte{\mathcal{O}}
\def\Spec{\mathrm{Spec}}
\def\Spf{\mathrm{Spf}}
\def\X{\mathcal{X}}
\def\A{\mathcal{A}}
\def\G{\mathcal{G}}
\def\fX{\mathfrak{X}}
\def\fZ{\mathfrak{Z}}
\def\Isom{\underline{\mathbf{Isom}}}
\def\wc{\mathbf{WCart}}
\def\W{\mathrm{W}}
\def\perf{\mathrm{perf}}
\def\Fil{\mathrm{Fil}}
\def\m{\mathfrak{m}}
\def\hol{\mathrm{H}}
\def\Rep{{\mathbf{Rep}}}
\def\Vect{{\mathbf{Vect}}}

\newcommand{\calA}{\mathcal{A}}
\newcommand{\calB}{\mathcal{B}}
\newcommand{\calC}{\mathcal{C}}
\newcommand{\calD}{\mathcal{D}}
\newcommand{\calE}{\mathcal{E}}
\newcommand{\calF}{\mathcal{F}}
\newcommand{\calG}{\mathcal{G}}
\newcommand{\calH}{\mathcal{H}}
\newcommand{\calI}{\mathcal{I}}
\newcommand{\calJ}{\mathcal{J}}
\newcommand{\calK}{\mathcal{K}}
\newcommand{\calL}{\mathcal{L}}
\newcommand{\calM}{\mathcal{M}}
\newcommand{\calN}{\mathcal{N}}
\newcommand{\calO}{\mathcal{O}}
\newcommand{\calP}{\mathcal{P}}
\newcommand{\calQ}{\mathcal{Q}}
\newcommand{\calR}{\mathcal{R}}
\newcommand{\calS}{\mathcal{S}}
\newcommand{\calT}{\mathcal{T}}
\newcommand{\calU}{\mathcal{U}}
\newcommand{\calV}{\mathcal{V}}
\newcommand{\calW}{\mathcal{W}}
\newcommand{\calX}{\mathcal{X}}
\newcommand{\calY}{\mathcal{Y}}
\newcommand{\calZ}{\mathcal{Z}}
\newcommand{\mathds}[1]{\mathbb{#1}}
\newcommand{\HT}{\mathrm{HT}}
\newcommand{\Zar}{\mathrm{Zar}}
\newcommand{\cou}{\mathrm{cou}}

\renewcommand{\pd}{\mathrm{pd}}
\newcommand{\nil}{\mathrm{nil}}
\newcommand{\wnil}{\mathrm{wnil}}

\pagestyle{fancy}
\fancyhf{}
\fancyhead[R]{\thepage}
\renewcommand{\headrulewidth}{0pt}
\maketitle

\abstract{Let $k$ be a perfect field of characteristic $p$, and let $X/k$ be a smooth variety. It is known that given a Frobenius lifting of $X$, we can identify prismatic crystals and nilpotent Higgs bundles, known as a positive characteristic version of the Simpson correspondence of $X$. However, Ogus--Vologodsky point out in their original paper of non-abelian Hodge theory in characteristic $p$ that, if we are just given a smooth lifting over $\W_2(k)$, there is a non-abelian Hodge theory on $p$-nilpotent Higgs bundles. Hence, it is natural to ask that whether there exists a subcategory of Hodge--Tate crystals on $X$, which can be described as $p$-nilpotent Higgs bundles. In this paper, we construct an analogue of the Hodge--Tate stack, so called the weakly $p$-nilpotent Hodge--Tate stack, on which the vector bundles are identified with certain Hodge--Tate crystals on $X$ that can be locally described by weakly $p$-nilpotent Higgs bundles. Furthermore, we prove that the weakly $p$-nilpotent Hodge--Tate stack is indeed a gerbe banded by $T_{X/k}\otimes{\alpha_p}$, and the obstruction class coincides with the obstruction of the existence of a Frobenius lifting of $X$, which is a conjecture of Bhatt and Lurie.}

\tableofcontents

\section{Introduction}

\subsection{Background}

The prismatic theory was introduced by Bhatt and Scholze (\cite{Bhatt_2022}) and developed by Bhatt and Lurie very recently (\cite{bhatt2022absolute}, \cite{bhatt2022prismatization}). It plays an important role in the recent study of $p$-adic Simpson theory and Sen theory (cf.\cite{bhatt2022absolute}, \cite{bhatt2022prismatization}, \cite{anschütz2023hodgetatestacksnonabelianpadic},\cite{anschütz2023smallpadicsimpsoncorrespondence},\newline\cite{min2024padicsimpsoncorrepondenceprismatic},\cite{Tian_2023},\cite{anschütz2022vvectorbundlespadicfields}, etc.). In this paper, we will focus the theory in the case of positive characteristic.

Let $k$ be a perfect field of characteristic $p$, $W=\mathrm{W}(k)$ be the Witt-vector of $k$, and $X/k$ be a smooth variety. Put $\mathrm{MIC}^{\mathrm{nil}}$ the category of integrable quasi--nilpotent connections and $\mathrm{Higgs}^{\mathrm{nil}}$ the category of nilpotent Higgs bundles. Recall that, given a flat lifting $\tilde{X}/(W/p^2)$ of $X$, Ogus and Vologodsky established in \cite{Ogus_2007} an equivalence between a full subcategory of $\mathrm{MIC}^{\mathrm{nil}}$ and a full subcategory of $\mathrm{Higgs}^{\mathrm{nil}}$. In addition, if we are given a lifting of the Frobenius of $X$ on $\tilde{X}$, there is an equivalence between $\mathrm{MIC}^{\mathrm{nil}}$ and $\mathrm{Higgs}^{\mathrm{nil}}$ (cf.\cite{lan2012inversecartiertransformexponential} \cite{Daxin_Xu_2019},\cite{shiho2012notesgeneralizationslocalogusvologodsky}).
In this case, this should be understood as the following commutative diagram:
\begin{equation}\label{eq:000}\xymatrix{\mathrm{Higgs}^{\mathrm{nil}}\ar@{-->}[d]\ar[rrr]^-{\text{Non--abelian Hodge}} &&&\mathrm{MIC}^{\mathrm{nil}}\ar[d]\\ \text{Hodge--Tate Crystals}\ar[rrr] &&&\text{Crystalline Crystals}}.\end{equation}
Here, all solid arrows are equivalence of categories, in which the bottom arrow is provided in \cite{Bhatt_2023} (when the underlying prism is perfect) or \cite{yu2024prismaticcrystalsschemescharacteristic} (for general crystalline prism). The top solid arrow exists when we are given a Frobenius lifting on $\tilde{X}$. The left dashed arrow is the unique arrow to make the diagram commutes. When we are given a $p$-complete and $p$-completely smooth lifting $\fX/W$ of $X$ together with a Frobenius lifting (not just on $W/p^2$), a direct construction of this functor is provided in \cite{Tian_2023} (for affine $X$) and \cite{ogus2024crystallineprismsreflectionsdiffractions} (for general $X$). If we merely have a smooth lifting $\tilde{X}/(W/p^2)$ as well as a Frobenius lifting, the direct construction is obtained in \cite{wang2024prismaticcrystalssmoothschemes}.

However, in the original paper \cite{Ogus_2007}, the authors only considered the $p$-nilpotent Higgs fields and the non-abelian Hodge theory may exist for $p$-nilpotent Higgs bundles when $X$ does not have a Frobenius lifting. Hence, it is natural to ask the following question.

\begin{question}\label{question}
    Can we define a full subcategory of Hodge--Tate crystals which are locally described as a $p$-nilpotent Higgs bundles?
\end{question}

In this paper, we partially answer the above question. Indeed, we can define a weaker condition, called the weakly $p$-nilpotency of a Higgs bundle. Such Higgs bundles can be identified with vector bundles on a stack. Moreover, we find that this stack is already mentioned in a conjecture of Bhatt and Lurie.

In the following, we define a \emph{Frobenius lifting} of $X$ as a pair $(\widetilde{X},\phi)$ where $\widetilde{X}$ is a smooth lifting of $X$ over $W_2$ and $\phi:\widetilde{X}\to \widetilde{X}$ is a lifting of the Frobenius endomorphism of $X$, which is compatible with the canonical Frobenius endomorphism of $W_2$. An isomorphism between two Frobenius liftings is an isomorphism between $W_2$-schemes which is compatible with $\phi$.

\subsection{The stacky approach, Hodge--Tate grebes and a conjecture of Bhatt--Lurie}

Let $X_{\Prism}$ be the relative prismatic site $(X/W)_{\Prism}$. Denote by $T_{X/k}$ the tangent bundle of $X$. In \cite{bhatt2022absolute} and \cite{bhatt2022prismatization}, Bhatt and Lurie constructed a stack $X^{\mathrm{HT}}$, so-called the Hodge-Tate gerbe of $X$, on which vector bundles are equivalent to the Hodge-Tate crystals on $X_{\Prism}$. This gerbe is banded by an fpqc-group $T_{X/k}^{\sharp}$. We refer \cite{bhatt2022absolute} and \cite{bhatt2022prismatization} for explicit definitions.

The nilpotent Higgs bundles are identified with vector bundles over $\mathbf{B}T_{X/k}^\sharp$. Consequently, a Simpson correspondence of $X$ is a equivalence between the category of vector bundles on $X^{\mathrm{HT}}$ and the category of vector bundles on $\mathbf{B}T^\sharp_{X/k}$, which is compatible with tensor products and sends $\mathbb{D}(\calV)$ to the zero Higgs bundle associated to $\calV$ for any vector bundle $\calV$ on $X$.

Actually, we expect that this equivalence comes from a trivialization (as a $T^\sharp_{X/k}$-gerbe) of $X^{\mathrm{HT}}$. To support this, we need to understand the relationship between Frobenius lifting and trivialization of $X^{\mathrm{HT}}$. Indeed, this was described as the following conjecture by Bhatt and Lurie (cf.\cite{bhatt2022prismatization}, Conjecture 5.14 (a)).

\begin{conj}[\cite{bhatt2022prismatization}]
    Let $o_{F,\mathrm{fpqc}}\in \hol^2(X_{\mathrm{fpqc}},\alpha_p\otimes T_X)$ be the obstruction of the existence of a Frobenius lifting. Let $o_\HT\in \hol^2(X_{\mathrm{fpqc}},T^\sharp_{X/k})$ be the cohomology class of $X^{\mathrm{HT}}$. Then the image of $o_\HT$ in $\hol^2(X_{\mathrm{fpqc}},\alpha_p\otimes T_X)$, under the map induced by $\mathds{G}_a^\sharp\to \alpha_p$ tensoring with $T_{X/k}$, is equal to $o_{F,\mathrm{fpqc}}$.
\end{conj}

We will prove this conjecture.

\begin{theo}\label{theo:main}
    The conjecture is true.
\end{theo}

Furthermore, we will prove the following.

\begin{theo}[Theorem \ref{theo: section X^HT to frob liftting}]
    There exists a canonical map from the set of equivalent classes of trivializations of $X^{\mathrm{HT}}$ and the set of isomorphic classes of Frobenius liftings of $X$.
\end{theo}

We will always work on the $\mathrm{fpqc}$ topology on $X$ rather than $\mathbf{fppf}$ topology since the group $T^\sharp_{X/k}$ is \textbf{not} of finite type. This site is too big that for a general presheaf $\mathcal{F}$, it may not have a sheafification. In particular, it is hard to talk about the tensor product of two sheaves of modules.

To due these problems, we follow the same method by Peter Scholze in \cite{scholze2022etale}. Precisely, we fix a very strong cardinal $\kappa$ such that the cardinal of points in the underlying space of $X$ is strictly smaller than $\kappa$ and the cardinal of $k$ is strictly smaller than $\kappa$.

\begin{defi}
    A $k$-scheme $Y$ is called $\kappa$-small if the cardinal of points in the underlying space of $Y$ is strictly smaller than $\kappa$ and for any affine open subset $\Spec(R)\subset Y$, the cardinal of $R$ is smaller than $\kappa$.
\end{defi}

We will use $X_{\mathrm{fpqc}}$ to denote the site of all $\kappa$-small schemes over $X$ as well as the $\mathrm{fpqc}$-covering. This is a essentially small category, so the general sheafification exists. By Yoneda's lemma, we have the following.

\begin{lemma}
    The functor $Y\to \underline{\mathrm{Hom}}_X(-,Y)$ defines a fully faithful functor from the category of $\kappa$-small schemes over $X$ to the category of sheaves on $X_{\mathrm{fpqc}}$.
\end{lemma}

We explain our proof of Theorem \ref{theo:main}. The point is to consider the following diagram
\begin{equation}\label{eq:basic obstruction sequence in H^1}
    \begin{tikzcd}
        0\ar{r} &\W(\calO_X)/F(\W(\calO_X))\ar{d} \ar{r}{V} & \W(\calO_X)/p\ar{r}\ar{d} & \calO_X\ar{d}\ar{r} &0\\
        0\ar{r} &F_*\calO_X/\calO_X\ar{r}{V} &\W_2(\calO_X)/p\ar{r}& \calO_X\ar{r} &0
    \end{tikzcd},
\end{equation}
in which $F$ (resp. $V$) is the Frobenius (resp. Verschiebung) and $\W$ (resp. $\W_2$) is the ring of Witt vectors (resp. truncated Witt vectors). All columns are natural projections. Moreover, we consider
\[\W(\calO_X)/F(\W(\calO_X))\]
as an $\calO_X$-module via the formula
\[a\cdot x:=[a^p]x\ \forall a\in \calO_X,x\in \W(\calO_X)/F(\W(\calO_X))\]
where $[a]$ is the Teichmuller lifting of $a$. Put $\widetilde o_{F}\in \hol^1\left(X_{\text{Zar}},T_{X}\otimes_{\calO_X}\W(\calO_X)/F\left(\W(\calO_X)\right)\right)$ as the obstruction of the existence of a split of the first line and $o_{F}\in \hol^1\left(X_{\text{Zar}},T_{X/k}\otimes_{\calO_X}F_*\calO_X/\calO_X\right)$ as the obstruction of the existence of a split of the second line. Obviously, $o_{\text{HT}}'$ projects to $o_{F}'$ under the natural projection.

Using the Grothendieck spectral sequence, we have a diagram
\[\begin{tikzcd}
    \hol^1\left(X_{\text{Zar}},T_{X/k}\otimes_{\calO_X}\W(\calO_X)/F\left(\W(\calO_X)\right)\right)\ar{d}\ar{r} &\hol^2\left(X_{\text{fpqc}},T_{X/k}^\sharp\right)\ar{d}\\
    \hol^1\left(X_{\text{Zar}},T_{X/k}\otimes_{\calO_X}F_*\calO_X/\calO_X\right)\ar{r} &\hol^2\left(X_{\text{fpqc}},T_{X/k}\otimes_{\calO_X}\alpha_p\right)
\end{tikzcd}.\]
We will prove that the image of $\widetilde{o}_F$ under the first line is equal to the cohomological class of $X^{\text{HT}}$ and the image of $o_F$ under the second line is equal to the obstruction class of the existence of Frobenius liftings. Theorem \ref{theo:main} follows from the commutativity of the above diagram.

\subsection{The weakly $p$-nilpotent Hodge--Tate stacks}

The projection of $o_{\HT}$ in $\hol^2(X_{\mathrm{fpqc}},T_{X/k}\otimes\alpha_p)$ can be described as a push-forward of $X^\HT$ along the homomorphism $T^\sharp_{X/k}\to T_{X/k}\otimes\alpha_p$. Hence, we can define a new stack $X^{\HT}_{p-\nil}$, in the following way.

\begin{defi}\label{defi: weakly p nil HT}
    Let $X^{\HT}_{p-\wnil}$ be the push--forward of $X^{\HT}$ along the homomorphism
    \[T^\sharp_{X/k}\to T_{X/k}\otimes\alpha_p.\]
\end{defi}

We find that this stack partially answers Question \ref{question}. Explicitly, we make the following definition.

\begin{defi}
    Let $(E,\theta)$ be a Higgs bundle on $X$. We call $(\calE,\theta)$ weakly $p$-nilpotent if for any open subset $U\subset X$ and any section $\partial\in T_{X/k}(U)$
    the action of $\partial^p$ on $\calE(U)$ is zero.
\end{defi}

\begin{theo}
    Assume $X$ has a $\delta$-lifting $X'$ over $\W(k)$. There exists an equivalence between the category of vector bundles on $X^\HT_{p-\wnil}$ and the category of weakly $p$-nilpotent Higgs bundles.
\end{theo}

\subsection{Notations}

\subsubsection*{ General Notations of sites, sheaves and stacks}

\begin{enumerate}
    \item Let $\X$ be a site and $\G$ be a stack on $\X$. For any $Y\in\X$ and $x_1,x_2\in\G(Y)$, denote by $\Isom_Y(x_1,x_2)$ the sheaf (on the restricted site $\X/Y$) of isomorphisms between $x_1$ and $x_2$.

\end{enumerate}

\subsubsection*{Notations of rings, schemes and formal schemes}
\begin{enumerate}
    \item For any ring $R$ of characteristic $p$, let $R^{\perf}$ be its perfection
$$R^{\perf}:=\varinjlim_{x\mapsto x^p}R$$
and obviously, we can globalize this definition to any scheme of characteristic $p$. For any scheme $Y$ over $\mathbb{F}_p$, define $Y^{\perf}$ as the perfection of $Y$.

\item For any ring $R$, we will use $\W(R)$ (resp. $\W_n(R)$) to denote the ring of Witt vectors (resp. truncated Witt vectors). More generally, for any ringed site $(\calX,\calO)$, we will use $\W(\calO)$ (resp. $\W_n(\calO)$) to denote the sheaf of rings sending each $U\in\calX$ to $\W(\calO(U))$ (resp. $\W_n(\calO(U))$ ).

\item For any scheme or formal scheme $Y$, $|Y|$ denotes the underlying space of $Y$.

\item Let $Y$ be a scheme, the ringed space $(|Y|,\W(\calO_{Y}))$ (resp. $(|Y|,\W_n(\calO_Y))$) is simply denoted by $\W(Y)$ (resp. $\W_n(Y)$). If $Y=\Spec(R)$ is affine and perfect, $\W(Y)$ (resp. $\W_n(Y)$) is equal to $\Spf(\W(R))$ (resp. $\Spec(\W_n(R))$). Generally, if $Y$ is a perfect scheme, $\W(Y)$ (resp. $\W_n(Y)$) is a formal scheme (resp. scheme).

\item Let $\calS ch$ be the category of all quasi-compact schemes. For a scheme $Y$, denote by $\calS ch_Y$ the category of all quasi-compact schemes over $Y$.

\item Let $f:Y\to S$ be a smooth morphism of schemes, define
\[T_{Y/S}\]
as the tengent bundle.

\item Let $Y$ be a scheme of characteristic $p$, define $F_*(-)$ (resp. $F^*(-)$) as the Frobenius push-forward (resp. Frobenius pull-back) functor of $Y$.

\end{enumerate}

\subsubsection*{Divided powers}

We refer \cite{Roby_1963} and \cite[\href{https://stacks.math.columbia.edu/tag/07H4}{Tag 07H4}]{stacks-project} for the following definitions.

\begin{enumerate}
    \item Let $R$ be a ring, denote by
    \[R[X_1,X_2,\dots,X_n]_{\pd}\]
    the divided power polynomial ring of $n$ variables over $R$.

    \item Let $R$ be a ring and $E$ be a finite rank projective $R$-module. Denote by
    \[R[E]_{\pd}\]
    the free $\pd$-algebra of $E$.
    Generally, for a scheme $Y$ and a vector bundle $\calE/Y$, we can define a quasi-coherent $\calO_Y$-algebra
    \[\calO_Y[\calE]_{\pd}.\]

    \item Let $R$ be a ring and $E$ be a finite rank projective $R$-module. Denote by
    \[\Gamma^i(E)\]
    the $i$-th divided power of $E$. By definition, it is canonically isomorphic to the module
    \[(\bigotimes_{j=1}^i E)^{S_i}\]
    where $S_i$ is the $i$-th permutation group that acts on $\bigotimes_{j=1}^i E$ by permuting the tensor factors.
\end{enumerate}

\subsubsection*{Various groups associated to vector bundles}

\begin{enumerate}
\item Let $R$ be a ring and $E$ be a finite rank projective module over $R$, denote by $R[E]$ the symmetric algebra of $E$. Generally, if $Y$ is a scheme and $\calE$ is a vector bundle over $Y$, denote by $\calO_Y[\calE]$ the symmetric algebra of $\calE$.

\item Let $Y$ be a scheme and $\calE$ be a vector bundle over $Y$. The functor sending all $(f:Y'\to Y)\in \calS ch_Y$ to the global section $\Gamma(Y',f^*\calE)$ is represented by $\Spec(\calO_Y[\calE^\vee])$. In the following, we will consider $\calE$ both as a Zariski sheaf and a presheaf of groups on $\calS ch_Y$.

    \item Let $\mathbb{G}_a^\sharp$ be the group
    \[\Spec(\Z[X]_{\mathrm{pd}}).\]
    For an affine scheme $Y=\Spec(R)$, $\mathbb{G}_{a}^\sharp( Y)$ has a natural $R$-module structure. Hence, we can view $\mathbb{G}_a^\sharp$ as a presheaf of $\mathbb{G}_a$-modules on $\calS ch$. 

    \item Let $Y$ be a scheme and $\calE$ be a vector bundle on $Y$, define the group functor $\calE^\sharp$ on $\calS ch_Y$ as 
    \[\mathbb{G}_a^\sharp\otimes_{\mathbb{G}_a}\calE.\]
    Note that since $\calE$ is locally free on $Y$, $\calE^\sharp$ is represented by the relative spectrum
    \[\Spec(\calO_Y[\calE^\vee]_{\pd}).\]
\end{enumerate}

\subsubsection*{Obstructions}

Let $k$ be a perfect field and $X/k$ be a smooth variety.
\begin{enumerate}
    \item Let $X^\HT$ be the Hodge--Tate stack of $X$. It is a fpqc $T_{X/k}^\sharp$-gerbe and hence associate an obstruction class
    \[o_\HT(X)\in \hol^2(X_{\mathrm{fpqc}},T_{X/k}^\sharp).\]

    \item There exists a cohomology class
    \[o_F(X)\in \hol^1(X_{\Zar},T_{X/k}\otimes_{\calO_X}F_*\calO_X/\calO_X)\]
    parameterizing the failure of the existence of a Frobenius liftings over $\W_2(k)$. By Grothendieck spectral sequence, it induces an element in $\hol^2(X_{\mathrm{fpqc}},\alpha_p\otimes_{\mathbb{G}_a}T_{X/k})$, which we will denoted by $o_{F,\mathrm{fpqc}}(X)$.
\end{enumerate}
We sometimes omit $X$ and just use $o_\HT$ (resp. $o_F$, $o_{F,\mathrm{fpqc}}$) to denote the corresponding obstruction class.

\subsection{Acknowledgement}

The author would like to thank Yupeng Wang for introducing this topic and helping the author on writing. We also thank Xiaoyu Qu for helpful discussion on some homological algebra.

\section{Frobenius liftings of smooth schemes}

One of the crucial objects in this paper is the Frobenius lifting. We would like to recall some basic properties of the Frobenius liftings. All results in this section are not new (see, for example, the appendix of \cite{Mehta87}.

The main consequence is Theorem \ref{theo:obstruction of lifting frob}, saying that the Frobenius liftings form a Zariski torsor of $T_{X/k}\otimes_{\calO_X}F_{*}\calO_X/\calO_X$, and is isomorphic to the torsor of splits of the square zero extension
\[\W_2(\calO_X)/p\to \calO_X.\]

Let $k$ be a perfect field, and $X/k$ be a smooth scheme.

\begin{defi}
    Let $n\geq 1$ be an integer. A Frobenius lifting of $X$ over $\W_n(k)$ consists of a smooth scheme $\widetilde{X}/\W_n(k)$, an isomorphism of $k$-schemes 
    \[X\cong \widetilde{X}\times_{\Spec\left(\W_n(k)\right)}\Spec(k)\]
    (hence, we can view $X$ as a closed subscheme of $\widetilde{X}$) and an automorphism (of schemes) $\phi$ of $\widetilde{X}$ satisfying the following properties.
    \begin{enumerate}
        \item The following diagram commutes
        \[\begin{tikzcd}
            \widetilde{X} \arrow{r}{\phi}\ar[d] & \widetilde{X}\ar[d]\\
            \Spec\left(\W_n(k)\right) \arrow{r}{\Spec(\varphi)} & \Spec\left(\W_n(k)\right)
        \end{tikzcd},\]
        where $\varphi$ is the Frobenius of Witt vectors.
        \item The restriction of $\phi$ on $X$ is equal to the Frobenius of $X$.
    \end{enumerate}

    Define isomorphisms between Frobenius liftings in the natural way. Similarly, we define a Frobenius lifting of $X$ over $\W(k)$ as a $p$-completely smooth formal scheme $\fX$ over $\W(k)$ which contains $X$ as the special fiber, and an endomorphism $\phi$ of $\fX$ that lifts the Frobenius of $X$ and $\phi$ is compatible with the Frobenius of $\W(k)$.

    If we do not specify the index $n$, it is always assumed that a Frobenius lifting is over $\W_2(k)$.
\end{defi}

\begin{rmk}
    Indeed, the compatibility assumption $1$ is not necessary, see \cite[Lemma 6.5.13]{Gabber_2003}
\end{rmk}

By the usual Zariski descent, there exists a stack $\calF$ on $X_{\Zar}$, sending any open subset $U\subset X$ to the groupoid of Frobenius liftings of $U$. The first important property of $\calF$ is the following.

\begin{theo}\label{theo:obstruction of lifting frob}
    The stack $\calF$ is a sheaf. Moreover, there exists a canonical $T_{X/k}\otimes_{\calO_X}F_*\calO_X/\calO_X$-action on $\calF$ that makes $\calF$ a $T_{X/k}\otimes_{\calO_X}F_*\calO_X/\calO_X$-torsor.
\end{theo}

Before proving the theorem, we need the following result.

\begin{lemma}\label{lemma:key lemma lift frob}
    Let $A$ be a flat $\Z/p^2$-algebra and $I:A\to A$ be an automorphism such that
    $$I(x)\equiv x\mod p$$
    for any $x\in A$. Let $\phi:A\to A$ be an endomorphism such that
    $$\phi(x)\equiv x^p\mod p$$
    for all $x\in A$. Put $\phi'=I\circ\phi\circ I^{-1}$. Then for any $x\in A$, there exists a $y\in A$ such that
    $$\phi(x)-\phi'(x)-py^p=0.$$

    Moreover, if $A/p$ is reduced, $I=id$ if and only if $\phi'=\phi$.
\end{lemma}

\begin{proof}
    Let $d=\frac{I-id_A}{p}:A/p^2\to A/p$. As for any $x\in A$, $I(px)-px=p(I(x)-x)=0$, $d$ induces a derivation from $A/p$ to itself. Hence,
    $$\phi'(x)=I\circ\phi\circ I^{-1}(x)=(id_A+pd)\circ \phi\circ(id_A-pd)(x).$$
    The right hand side is equal to
    $$\phi\left(x-pd(x)\right)+pd\left(\phi\left(x-pd(x)\right)\right)\equiv \phi(x)-pd(x)^p\mod p$$
    where the last equality follows from the fact $d\left(\phi(z)\right)\equiv d(z^p)\equiv pz^{p-1}d(z)\mod p$ for any $z$. Thus
    $$\phi(x)-\phi'(x)-pd(x)^p=0$$ and the first part of the lemma follows.

    Assume $A/p$ is reduced. The above calculation shows that $\phi=\phi'$ is equal to $d=0$, which is also equal to $I=id$. 
\end{proof}

\begin{proof}[Proof of Theorem \ref{theo:obstruction of lifting frob}]
    The sheafiness of $\calF$ follows from Lemma \ref{lemma:key lemma lift frob}, which says that the automorphic groups of all sections are trivial.
    
    Now we construct an action of $T_{X/k}\otimes_{\calO_X}F_*\calO_X/\calO_X$ on $\calF$. For any affine open $U=\Spec(A)$ of $X$, a Frobenius lifting $(\widetilde{A},\phi)$ and $\bar \partial\in \mathrm{Der}_{k}(A,F_*A/A)$, define the action of $\bar{\partial}$ on $(\widetilde{A},\phi)$ as $(\widetilde A,\phi+p\partial)$ where $\partial$ is a lifting of $\bar\partial$ in $\mathrm{Der}_k(A,F_*A)$. This construction does not depend on the choice of $\partial$ or the isomorphic class of $(\widetilde{A},\phi)$ according to the calculation in the proof of Lemma \ref{lemma:key lemma lift frob}. Moreover, this action provides a bijection between $\mathrm{Der}_k(A,F_*A/A)$ and the set of isomorphic classes of Frobenius liftings. Indeed, since $A$ is smooth, each $\W_2(k)$-smooth lifting of $A$ is isomorphic to $\widetilde{A}$. Thus, the set of isomorphic classes of Frobenius liftings of $A$ is identified with the set of Frobenius liftings on $\widetilde{A}$ modulo the action of the automorphic group of the thickening $(\widetilde{A}\rightarrow A)$.

    By usual deformation theory, the Frobenius liftings on $\widetilde{A}$ are identified with
    \[\mathrm{Der}_k(A,F_*A),\]
    and the automorphisms of $(\widetilde{A}\to A)$ are identified with
    \[\mathrm{Der}_k(A,A).\]
    The calculation in the proof of Lemma \ref{lemma:key lemma lift frob} shows the claim.
\end{proof}

\begin{defi}
    We call $\calF$ the \emph{torsor of Frobenius liftings} of $X$ and the corresponding cohomology class $o_F\in \hol^1(X_{\text{Zar}},T_{X/k}\otimes_{\calO_X}F_*\calO_X/\calO_X)$ is called the \emph{obstruction class of the existence of Frobenius liftings}.
\end{defi}

As a direct corollary, we have the following.

\begin{theo}
    Let $X/k$ be a smooth variety and $o_F\in \hol^1(X_{\text{Zar}},T_{X/k}\otimes_{\calO_X}F_*\calO_X/\calO_X)$ be the obstruction class of the existence of Frobenius liftings of $X$. Then the existence of Frobenius liftings of $X$ is equivalent to $o_F=0$. Moreover, if this holds, the set of isomorphic classes of Frobenius liftings of $X$ admits an action by  $\hol^0(X_{\text{Zar}},T_{X/k}\otimes_{\calO_X}F_*\calO_X/\calO_X)$, which is transitive and faithful.
\end{theo}

\begin{proof}
    Let $\calF$ be the torsor of Frobenius liftings. By definition, the existence of Frobenius liftings is equal to $\Gamma(X,\calF)\neq \emptyset$, which is equal to the triviality of the torsor $\calF$. This proves the first claim. If $\calF$ is trivial, the second claim is valid by the definition of torsors.
\end{proof}

\begin{theo}
    The sheaf of rings $\W_2(\calO_X)/p$ is a square zero extension of $\calO_X$ by the sheaf of modules $F_*\calO_X/\calO_X$. In particular, there exists a canonical exact sequence
    \[0\to F_*\calO_X/\calO_X\to \W_2(\calO_X)/p\to \calO_X\to 0.\]
\end{theo}

\begin{proof}
    Define the homomorphism $\W_2(\calO_X)\to \calO_X$ as the projection to the first ghost coordinate. It is obviously surjective. Let $\pi:\W_2(\calO_X)/p\to \calO_X$ be the induced homomorphism.

    First, we check that $(\ker \pi)^2=0$, which is equivalent to checking that for any
    \[x,\ y\in \W_2(\calO_X),\]
    \[V(x)V(y)\in p\W_2(\calO_X).\]
    This follows from the equality
    \[V(x)V(y)=V(FV(x)y)=V(pxy)=pV(xy).\]

    Define the morphism of sheaves $g:\calO_X\to \ker \pi$ as sending $f$ to the projection of the ghost coordinate $(0,x)$. Since in $\W_2(\calO_X)$
    \[(0,x)+(0,y)=(0,x+y)\]
    $g$ is a homomorphism of groups.
    By definition, for any ring $R$ over $\mathbb{F}_p$ and $(x,y)\in \W_2(R)$,
    \[p(x,y)=(0,x^p).\]
    Hence, $g$ induces an isomorphism of groups $F_*\calO_X/\calO_X\cong \ker \pi$. After a direct calculation, one can prove that this is $\calO_X$-linear.
\end{proof}

Define $o\in \hol^1(X,T_{X/k}\otimes_{\calO_X}F_*\calO/\calO)$ as the obstruction of trivializing the extension $\W_2(\calO_X)/p\to \calO_X$.

\begin{theo}\label{theo: splitting W2=lifting}
    Let $X/k$ be a smooth variety.
    \begin{enumerate}
        \item There exists a canonical map from the set of the retractions of
        \[\W_2(\calO_X)/p\to \calO_X\]
        to the set of isomorphic classes of Frobenius liftings.

        \item Let $o_{F}$ be the obstruction class of the existence of $F$-liftings. Then $o_F=o$ in \[\hol^1(X,T_{X/k}\otimes_{\calO_X}F_*\calO/\calO).\]
    \end{enumerate}
\end{theo}

\begin{proof}
    (1) Let $R$ be a ring of characteristic $p$, and let $(\widetilde{R},\phi)$ be a $F$-lifting of $R$. Define
    \[\delta:\widetilde{R}\to R\]
    as $\frac{\phi(x)-x^p}{p}$. It is easy to prove that $\delta(px+y)=x^p+\delta(y)$ for any $x\in R$. Hence $\delta$ induces a map from $R$ to $F_*R/R$. It is easy to check that
    \[x\mapsto (x,\delta(x))\]
    induces a homomorphism from $R$ to $\W_2(R)/p$.

    This defines a map from the torsor of Frobenius liftings and the torsor of the retractions of $\W_2(\calO_X)/p\to \calO_X$. One can check that this map is $T_{X/k}\otimes_{\calO_X}F_*\calO_X/\calO_X$-equivariant. Hence, this map is an isomorphism of sheaves. The claim then follows from taking global sections.

    (2) Let $\calF$ be the Zariski sheaf of Frobenius liftings, which is a $T_{X/k}\otimes_{\calO_X}F_*\calO_X/\calO_X$-torsor. By definition, the cohomology class of this torsor is equal to the obstruction of lifting Frobenius.
\end{proof}

\section{The Hodge--Tate gerbe and its cohomology class}

\subsection{Obstruction of a gerbe}

We refer \cite[\href{https://stacks.math.columbia.edu/tag/06NY}{Tag 06NY}]{stacks-project} for the reference of this section.

Let $\X$ be a site containing a final object $X\in \X$. Let $\A$ be a sheaf of abelian groups on $\X$.

\begin{defi}
    A gerbe on $\X$ banded by $\A$ (which will be called an $\A$-gerbe in the following) is a non-empty (i.e. there is a covering of $\{U_\lambda\to X:\lambda\in \Lambda\}$ such that the sections $\G(X_\lambda)$ is not empty) stack $\G$ on $\X$ as well as, for any $Y\in\X$ and $x\in \G(Y)$, an isomorphism of sheaves
    $$\varphi_{x}:\A|_{Y}\cong \Isom_{Y}(x,x).$$ Satisfying the condition: For any $Y$, any sections $x_1,x_2\in \G(Y)$ and any isomorphism $i:x_1\cong x_2$, the following diagram commutes.
    $$\xymatrix@C+30pt{\A|_{Y}\ar[r]^-{id}\ar[d]_-{\varphi_{x_1}} &\A|_Y\ar[d]^-{\varphi_{x_2}}\\ \Isom_Y(x_1,x_1)\ar[r]^-{f\mapsto i\circ f\circ i^{-1}}&\Isom_Y(x_2,x_2)}$$
\end{defi}

\begin{lemma}\label{lemma:iso inv gerbes}
    Keep the notations above. For any $Y$ and two elements $x_1,x_2\in\G(Y)$, there is a canonical isomorphism
    $$\iota_{x_1\to x_2}:\Isom_Y(x_1,x_1)\cong\Isom_Y(x_2,x_2)$$ such that if there is an isomorphism $i:x_1\cong x_2$, then
    $$\iota_{x_1\to x_2}=(f\mapsto i\circ f\circ i^{-1}).$$
\end{lemma}

\begin{proof}
    Choose $\iota_{x_1\to x_2}$ to be $\varphi_{x_2}\circ\varphi_{x_1}^{-1}$.
\end{proof}

Let us recall the definition of the cohomology class of a gerbe. Abstractly speaking, a $\A$-gerbe is an $\mathbf{B}\A$-torsor where $\mathbf{B}\A$ is the stacky quotient $*/\A$. As $\A$ is an abelian sheaf, $\mathbf{B}\A$ is an abelian stack and thus an $\A$-gerbe $\G$ corresponds a cohomology class in $\alpha_G\in \hol^1(\X,\mathbf{B}\A)$. By the long exact sequence associated to
$$0\to \A\to 0\to \mathbf{B}\A\to 0,$$ there is a coboundary map $$\hol^1(\X,\mathbf{B}\A)\to \hol^2(\X,\A).$$
The associated cohomology class of $\G$ is defined to be the image of $\alpha_G$ under this coboundary map.

We provide an elementary construction of the class $o_\G$.

Suppose there exists a covering $\{f_{\lambda}:U_{\lambda}\to X:\lambda\in\Lambda\}$ and for any $\lambda\in\Lambda$, $\G(U_\lambda)\neq \emptyset$ and that for any $\lambda\in\Lambda$, we can choose a $x_{\lambda}\in \G(U_\lambda)$ such that for each pair $(\lambda,\mu)\in \Lambda^2$, $$x_\lambda\cong x_\mu$$ in $\G(U_\lambda\times_XU_\mu)$. For any pair $(\lambda,\mu)$, we fix an isomorphism $\psi_{\lambda\to \mu}:x_\lambda\xrightarrow{\cong} x_{\mu}$.

Now for any $(\lambda,\mu,\nu)\in\Lambda^3$, we define an automorphism
$$\psi_{\lambda,\mu,\nu}:x_\lambda\xrightarrow{\psi_{\lambda\to\mu}}x_\mu\xrightarrow{\psi_{\mu\to \nu}}x_{\nu}\xrightarrow{\psi_{\lambda\to \nu}^{-1}}x_\lambda$$
of $x_{\lambda}|_{U_\lambda\times_XU_{\mu}\times_XU_{\nu}}$. Let $\sigma_{\lambda,\mu,\nu}$ be the image of $\psi_{\lambda,\mu,\nu}$ in $\A(U_\lambda\times_XU_{\mu}\times_XU_{\nu})$ under the banding map $\varphi_{x_{\lambda}}$.

\begin{prop}\label{cechgerbe}
    The class $\{-\sigma_{\lambda,\mu,\nu}\}$ defines a \v{C}ech cocycle in $C^2(\{U_\lambda\to X\},\A)$ and its image in $\hol^2(\X,\A)$ is equal to the class $o_\G$.
\end{prop}

\begin{proof}
    This comes from the fact that the \v Cech coboundary map and coboundary map are equal. See \cite{Breen_2009} or \cite[Section 2]{AST_1994__225__1_0}.
\end{proof}

\begin{lemma}\label{uniqueiso}
    Keep the notations, for any $\A$-torsor $x$ over $\X$, there is a unique isomorphism of $\A$-gerbes
    $$\tau_{x}:\mathbf{B}\A\cong \mathbf{B}\A$$
    such that $\tau_{x}$ sends the trivial torsor on $\X$ to $x$.
\end{lemma}

\begin{proof}
    See \cite[Section 1]{AST_1994__225__1_0}.
\end{proof}

\subsection{The Hodge--Tate gerbes and a first reduction of Theorem \ref{theo:main}}\label{recol}

Recall we fix a perfect field $k$ of characteristic $p$, $W=\mathrm{k}$, and a smooth variety $X/k$. We first recall some basic definitions.

\begin{defi}
    The $\mathrm{fpqc}$-sheaf $\mathds{G}_a^{\sharp}$ is the sheaf on $X_{\mathrm{fpqc}}$ sending any $\Spec(B)\to X$ to
    $$\ker\left(F:\W(B)\to \W(B)\right).$$
\end{defi}

\begin{rmk}
    For any ring $B$, since $F:\W(B)\to \W(B)$ is a ring homomorphism, $\mathds{G}_a^\sharp(B)=\ker(F)$ is an ideal in $\W(B)$. In particular, $\mathds{G}_a^\sharp(B)$ is a $\W(B)$-module. Since for any $x\in \mathds{G}_a^\sharp(B)=\ker(F)$ and $y\in \W(B)$ we have
    $$xV(y)=V\left(F(x)y\right)=0,$$
    $\mathds{G}_a^\sharp(B)=\ker(F)$ is a $\W(B)/V\W(B)\cong B$-module via the isomorphism $$\W(B)/V\W(B)\cong B.$$
\end{rmk}

In \cite{bhatt2022prismatization} Chapter 5, they define the Hodge-Tate gerbe to be the stack on $X_{\mathrm{fpqc}}$, sending any $f:\Spec(B)\to X$ to the groupoid of the diagram of schemes over $k$
$$\xymatrix{\Spec(B)\ar[dr]_-{f} \ar[r] &\Spec\left(\W(B)/p\right)\ar[d]^-{g}\\ &X}$$
where $\W(B)/p$ is the derived quotient, $\Spec(B)\to \Spec\left(\W(B)/p\right)$ is the morphism induced by the projection to the first ghost coordinate $\W(B)\to B$ and the structure map $\Spec\left(\W(B)/p\right)\to\Spec(k)$ is given by $\W(k)\to \W(B)$ (the functorality of witt vectors) modulo $p$. The isomorphisms are provided by the isomorphisms of Cartier divisors
\[\W(B)\xrightarrow[]{p} \W(B).\]
We denote this stack by $X^{\mathrm{HT}}$, which admits a natural morphism to $X$. The above definition, however, is highly derived and non-explicit. There is also an explicit description of this stack, which is a quotient of $X^{\perf}$ by a groupoid. This description was first introduced by Drinfeld in \cite{Drinfeld2022} for studying crystalline crystals, and then in \cite[Remark 7.9]{bhatt2022prismatization}, Bhatt and Lurie showed that this construction is equal to $X^{\text{HT}}$ (up to some Frobenius twists).

Let $X^{\perf}\to X$ be the perfection of $X$, and $\mathbb{X}^{\perf}$ be the perfect prism of $X^\perf$. Consider the relative prismatic cohomology \[\Prism_{(X^\perf\times_{X}X^\perf)/(\mathbb{X}^\perf\times\mathbb{X}^\perf)}.\]
By \cite[Exmaple 7.9]{Bhatt_2022}, it concentrates to the degree $0$ and there exists a canonical morphism
\[ \overline\Prism_{(X^\perf\times_{X}X^\perf)/(\mathbb{X}^\perf\times\mathbb{X}^\perf)}:=\Prism_{(X^\perf\times_{X}X^\perf)/(\mathbb{X}^\perf\times\mathbb{X}^\perf)}/p\to X^\perf\times_{X}X^\perf\]
which is flat. Let $p_1$ and $p_2$ be two induced projections
\[\overline\Prism_{(X^\perf\times_{X}X^\perf)/(\mathbb{X}^\perf\times\mathbb{X}^\perf)}\rightrightarrows X^{\perf}.\]
Then the above diagram forms a groupoid in the category of schemes and $X^{\mathrm{HT}}$ is equal to the quotient stack of this diagram. 

Indeed, $\overline\Prism_{(X^\perf\times_{X}X^\perf)/(\mathbb{X}^\perf\times\mathbb{X}^\perf)}$ is a  $T_{X/k}^\sharp$-torsor over $X^\perf\times_XX^\perf$, and the $T_{X/k}^{\sharp}$-gerbe structure of $X^{\mathrm{HT}}$ comes from this $T_{X/k}^\sharp$-action. The definition of the action of $T_{X/k}^\sharp$ is not trivial. However, this fact tells us that there should be a $T^{\sharp}_{X/k}$-torsor on $X^\perf\times_XX^\perf$, which plays an important role during the proof of Theorem \ref{theo:main}. Hence, instead of using Drinfeld's construction directly, we will construct a $T^{\sharp}_{X/k}$-torsor on $X^\perf\times_XX^\perf$ by defining an element in $\hol^1\left((X^\perf\times_XX^\perf)_{\text{fpqc}},T_{X/k}^\sharp\right)$. Before our construction, let us recall some basic cohomological properties of $\mathds{G}_a^\sharp$.

\begin{lemma}\label{wittcohomology}
    Let $R$ be a ring of characteristic $p$ and $\W(\inte)$ be ring of Witt vectors of the $\mathrm{fpqc}$-structure sheaf of $\Spec(R)$. Then for any $i>0$
    $$\hol^i\left(\Spec(R),W(\inte)\right)=0.$$
\end{lemma}

\begin{proof}
    This is proved by \cite[Lemma 3.18]{scholze2012padic}.
\end{proof}

\begin{lemma}\label{isom torsor}
    For any smooth $k$-scheme $X$, there is a canonical isomorphism
    $$R^{1}\nu_{*}T^\sharp_{X/k}\cong \W(\inte_X)/F\left(\W(\inte_X)\right)\otimes_{\inte_X} T_{Y/k},$$ where $\nu:X_{\mathrm{fpqc}}^{\sim}\to X_{\mathrm{Zar}}^{\sim}$ is the natural projection. Moreover, if $X=\Spec(R)$ is affine, we also have
    $$\hol^1(X_{\mathrm{fpqc}},T^\sharp_{X/k})\cong \W(R)/F\left(\W(R)\right)\otimes_{R} T_{X/k}.$$
    Here, the module structures of $\W(\inte_X)/F\left(\W(\inte_X)\right)$ and $\W(R)/F\left(\W(R)\right)$ are given in the proof to Lemma \ref{dercalculation}.
\end{lemma}

\begin{proof}
    We only need to prove the second claim. There is a homomorphism
    $$ \hol^1(Y_{\mathrm{fpqc}},\mathds{G}_a^{\sharp})\otimes_{R} T_{Y/k}\to \hol^1(Y_{\mathrm{fpqc}},T^\sharp_{Y/k})$$
    given by cup product. Since tensoring with a projective module preserves exactness and acyclicity (since any projective module is a direct summand of a free module). This homomorphism is isomorphic.

    We only need to show that
    $$\hol^1(Y_{\mathrm{fpqc}},\mathds{G}_a^{\sharp})\cong \W(R)/F\left(\W(R)\right)$$
    We consider the exact sequence
    $$0\to \mathds{G}_a^{\sharp}\to \W(\inte{})\xrightarrow{F} \W(\inte{})\to 0.$$
    Taking long cohomology sequence and by Lemma \ref{wittcohomology}, we get the required isomorphism as abelian groups. It remains to prove this isomorphism is $R$-linear. For any element $a\in R$, then we have a commutative diagram of short exact sequences
    $$\xymatrix{0 \ar[r] &\mathds{G}_a^{\sharp}\ar[d]^-{m\mapsto am}\ar[r] &\W(\inte{})\ar[r]^-{F}\ar[d]^{n\mapsto [a]n}&\W(\inte{})\ar[r]\ar[d]^{n\mapsto [a]n} &0\\
        0 \ar[r] &\mathds{G}_a^{\sharp}\ar[r] &\W(\inte{})\ar[r]^-{F}&\W(\inte{})\ar[r] &0}.$$ Take long cohomology sequence and the linearity follows.
\end{proof}

In fact, the above proof also implies the following stronger statement.

\begin{lemma}\label{strongcoh}
    Let $Y=\Spec(R)$ be affine and $\mathcal{E}$ be a vector bundle on $Y$ associated to the projective module $E/R$. Then
    $$\hol^1(Y_{\mathrm{fpqc}},\mathds{G}_a^{\sharp}\otimes_{\inte{}}\mathcal{E})\cong \W(R)/F\left(\W(R)\right)\otimes_{R}E$$
    and for any $i>1$,
    $$\hol^{i}(Y_{\mathrm{fpqc}},\mathds{G}_a^{\sharp}\otimes_{\inte{}}\mathcal{E})=0.$$
\end{lemma}

\begin{lemma}\label{coh0}
    Let $R$ be a $k$-algebra such that the Frobenius is surjective and $\mathcal{E}$ be a vector bundle on $\Spec(R)$, then $$\hol^i\left(\Spec(R)_{\mathrm{fpqc}},\mathds{G}_a^{\sharp}\otimes_{\inte_{\Spec(R)}}\mathcal{E}\right)=0$$ for any $i>0$
\end{lemma}

\begin{proof}
    This comes from the lemma \ref{strongcoh}.
\end{proof}

For any $n\geq 1$, let $X^{(n)}$ be the fiber product of $n$ copies of $X^{\perf}$ over $X$ and $f^{(n)}:X^{(n)}\to X$ be the structure morphism. Let $T^{\sharp,(n)}$ be the sheaf $f^{(n)}_*T^\sharp_{X/k}$. If $n=0$, let $X^{(n)}=X$. Then we have a cosimplicial sheaf $(T^{\sharp,(\bullet)})$ over $X$. Let $(T^{\sharp,(\bullet)},d)$ be the corresponding \v Cech sequence, which is exact by $\mathrm{fpqc}$-descent of vector bundles. For each $n$, let $\mathcal{Z}^{(n)}$ be the subsheaf $\ker(d)\subset T^{\sharp,(n)}$. Hence, we have an exact sequence of \textrm{fpqc}-sheaves
\begin{equation}\label{eq:exact T T1 Z2}
    0\to T^{\sharp}_{X/k}\to T^{\sharp,(1)}\to \mathcal{Z}^{(2)}\to 0.
\end{equation}

Our cohomology class in $\hol^1(X^{(2)}_{\text{fpqc}},T^{\sharp}_{X/k})$ will be an image in $\hol^1(X^{(2)}_{\text{Zar}},\calZ^{(2)})$.

\begin{lemma}\label{lemma: T^sharp on X^perf times_X X^perf}
    Let $X/k$ be a smooth variety, then for any vector bundle $\calE$ on $X$, Then the followings hold.
    \begin{enumerate}
        \item There are isomorphisms \[\hol^1(X^{(2)}_{\text{fpqc}},\calE^\sharp)\cong \hol^1(X^{(2)}_{\text{Zar}},\calE^\sharp)\cong \hol^1(X_{\text{Zar}},f^{(2)}_*\calE^\sharp).\]
    Both isomorphisms come from the Grothendieck spectral sequences.
    \item The derived push-out $R\nu_*(f^{(1)}_{*}\calE^\sharp)=0$.
    \end{enumerate}
\end{lemma}

\begin{proof}
    (1) The first isomorphism comes from Lemma \ref{coh0}. The second isomorphism follows from the fact that $X^{(2))}\to X$ is affine and that the restriction of $\calE^\sharp$ on $X^{(2)}_{\text{Zar}}$ is quasi-coherent.

    (2) We only need to prove that for any $i\geq 0$, 
    \[R^i\nu_*(f^{(1)}_*\calE^\sharp)=0.\]
    By definition, the sheaf $R^i\nu_*(f^{(1)}_*\calE^\sharp)$ is equal to the sheafification of teh presheaf sending each affine open $U\subset X$ to $\hol^{i}\left(U_{\text{fpqc}},f_*^{(1)}\calE^\sharp\right).$ We can choose $U=\Spec(A)$ small enough so that $\calE|_{U}$ is free. Then $f_*^{(1)}\calE^\sharp$ is a direct sum of $f_*^{(1)}\mathbb{G}_a^\sharp$, which sends any $\{\Spec(B)\to U\}$ to \[\ker (F:\W(A^{\perf}\otimes_{A}B)\to \W(A^{\perf}\otimes_{A}B)).\]
    Consider the exact sequnce
    \[0\to f^{(1)}_*\mathbb{G}_a^\sharp\to \W(A^{\perf}\otimes_A\mathbb{G}_a)\xrightarrow[]{F} \W(A^\perf\otimes_{A}\mathbb{G}_a)\to 0\]
    on $U_{\text{fpqc}}$, the same argument as Lemma \ref{isom torsor} shows that
    \[\hol^{i}(U_{\text{fpqc}},f^{(1)}_*\mathbb{G}_a^\sharp)=0\]
    for all $i\geq0$.
\end{proof}

Let $\nu$ be the projection from the fpqc-topology to the Zariski topology of $X$. As $R\nu_*T^{\sharp,(1)}=0$ (Lemma \ref{lemma: T^sharp on X^perf times_X X^perf}), there exists an isomorphism $\nu_*\calZ^{(2)}\cong R^1\nu_*T_{X/k}^\sharp$. Combining with Lemma \ref{strongcoh}, there exists an isomorphism \[\gamma:\nu_*\calZ^{(2)}\cong \W(\inte_X)/F\left(\W(\inte_X)\right)\otimes_{\inte_X} T_{X/k}.\]

We calculate $\gamma$ explicitly.

Let $U=\Spec(A) \subseteq X$ be an affine open subset. By definition
$$\calZ^{(2)}(U)=\fZ^{(2)}(A):=\{x\in T_{A/k}^\sharp(A^{\perf}\otimes_{A}A^{\perf}):\iota_{1,2}(x)+\iota_{2,3}(x)=\iota_{1,3}(x)\}$$ where $\iota_{1,2}$ (resp. $\iota_{2,3},\ \iota_{1,3}$) is the map $$T_{A/k}^\sharp(A^{\perf}\otimes_{A}A^{\perf})\to T_{A/k}^\sharp(A^{\perf}\otimes_{A}A^{\perf}\otimes_{A}A^{\perf})$$ induced by $$A^{\perf}\otimes_{A}A^{\perf}\to A^{\perf}\otimes_{A}A^{\perf}\otimes_{A}A^{\perf}:x\otimes y\mapsto x\otimes y\otimes1\text{ (resp. $1\otimes x\otimes y$, $x\otimes1\otimes y$)}.$$

\begin{lemma}\label{iso1}
    The first \v Cech cohomology group of $T^\sharp_{A/k}$ with respect to the \textrm{fpqc}-covering $\Spec(A^{\perf})\to\Spec(A)$ is isomorphic to $\fZ^{(2)}(A)$. More precisely, the projection $$\fZ^{(2)}(A)\to \check{\hol}^1\left(\{\Spec(A^{\perf})\to\Spec(A)\},T^{\sharp}_{A/k}\right)$$ is an isomorphism.
\end{lemma}

\begin{proof}
    This is equal to prove the subgroup of coboundaries
    $$\mathfrak{B}(A):=\{\iota_1(x)-\iota_2(x):x\in T^{\sharp}_{X/k}(A^{\perf})\}$$
    is equal to zero, where $\iota_1:T_{X/k}^\sharp(A^{\perf})\to T_{X/k}^\sharp(A^{\perf}\otimes_{A}A^{\perf})$ (resp. $\iota_2$) is induced by $A^\perf\to A^{\perf}\otimes_{A}A^{\perf},\ x\mapsto x\otimes1$ (resp. $1\otimes x$).

    This is obvious since $$\mathds{G}_a^{\sharp}(A^{\perf})=0.$$
\end{proof}

Let $\rm{Z}(A)=\{x\in \W(A^{\perf}\otimes_AA^{\perf}):F(x)=0,\ \iota_{1,2}(x)+\iota_{2,3}(x)=\iota_{1,3}(x)\}$. By definition, we have
\[\fZ^{(2)}(A)=\mathrm{Z}(A)\otimes_{A}T_{A/k}.\]

\begin{lemma}
    There is an isomorphism $$\gamma_0:\W(A)/F\left(\W(A)\right)\cong \mathrm{Z}(A)$$ such that $$\gamma_0([x])= 1\otimes F^{-1}(x)-F^{-1}(x)\otimes1.$$
    Moreover, it is compatible with cohomologies in the sense that the following diagram commutes,
    $$\xymatrix{\check{\hol}^1(\{\Spec(A^{\perf})\to \Spec(A)\},\mathds{G}_a^\sharp)\ar[r] & \hol^{1}(U_{\text{fpqc}},\mathds{G}_a^\sharp)\\ \mathrm{Z}(A)\ar[u]^-{\text{ projection}} & \W(A)/F\left(\W(A)\right)\ar[l]_-{\gamma_0}\ar[u] }$$
    where the right column comes from Lemma \ref{strongcoh}.
\end{lemma}

\begin{proof}
    The top arrow is isomorphic by \v Cech-derived spectral sequence and the lemma \ref{coh0}. The left arrow is zero since $\mathds{G}_a^\sharp(A^{\perf})=0$. The explicit formula of $\gamma$ also comes from the explicit formula of the coboundary map.
\end{proof}

\begin{cor}\label{cor:explicit gamma}
    The isomorphism \[\gamma(U):\Gamma(U,\nu_*\calZ^{(2)})\cong \Gamma\left(U,\W(\inte_X)/F\left(\W(\inte_X)\right)\otimes_{\inte_X} T_{X/k}\right)\] is equal to $\gamma_0^{-1}\otimes id$.
\end{cor}

\begin{proof}
    This comes from the definition of $\gamma$.
\end{proof}

\begin{lemma}\label{lemma: diagram change 2 to 1}
    The following diagram commutes,
    $$\xymatrix{\hol^1\left(X_{\mathrm{Zar}},\nu_*\mathcal{Z}^{(2)}\right)\ar[d]\ar[rrr]^-{\partial'}\ar[d]&&&\hol^1\left(X_{\mathrm{Zar}},R^1\nu_*T^{\sharp}_{X/k}\right)\ar[d]\\ \hol^1(X_{\mathrm{fpqc}},\mathcal{Z}^{(2)})\ar[rrr]^-{\partial}&&&\hol^2(X_{\mathrm{fpqc}},T^{\sharp}_{X/k})}$$
    in which the rows are cobundary and the columns come from the Grothendieck spectral sequences.
\end{lemma}

Before proving Lemma \ref{lemma: diagram change 2 to 1}, we recall some facts of homological algebra.

Let $\mathcal{A},\mathcal{B},\mathcal{C}$ be three abelian categories and assume $\mathcal{A}$ and $\mathcal{B}$ has enough injective objects. Let $F:\mathcal{A}\to \mathcal{B}$ and $G:\mathcal{B}\to \mathcal{C}$ be two left exact functors such that $F$ maps all injective objects to $G$-acyclic objects. Then for all objects $X\in \mathcal{A}$, we have the Grothendieck spectral sequence $E_r^{p,q}(X)$
$$E_2^{p,q}=R^{p}G\left(R^qF(X)\right)\Rightarrow R^{p+q}(G\circ F)(X).$$

\begin{theo}
    Assume we have an exact sequence
    $$0\to X\to Y\to Z\to 0.$$
    There is a canonical homomorphism of Grothendieck spectral sequences
    $$\delta_r^{p,q}:E_r^{p,q}(Z)\to E_r^{p,q+1}(X)$$
    such that:

    (1) $\delta_2^{p,q}:E_2^{p,q}(Z)=R^{p}G\left(R^qF(Z)\right)\to R^{p}G\left(R^{q+1}F(X)\right)=E_2^{p,q+1}(X)$ is equal to the coboundary map induced by the exact sequence.

    (2) $\delta_{\infty}^{p,q}:F^pE_\infty^{p+q}=F^pR^{p+q}(G\circ F)(Z)\to F^pR^{p+q+1}(G\circ F)(X)=F^pE_\infty^{p+q}$ is equal to the coboundary map induced by the exact sequence.
\end{theo}

\begin{proof}
    See \cite[Theorem 1.1]{Baraglia_2013}.
\end{proof}

This has a direct corollary:

\begin{cor}\label{coboundspec}
    Keep the notations in the above theorem. If $F(X)=0$, we have a commutative diagram
    $$\xymatrix{R^1G\left(F(Z)\right)\ar[d]\ar[rrr]^-{\text{coboundary}}\ar[d]&&&R^1G\left(R^1F(X)\right)\ar[d]\\ R^1(G\circ F)(Z)\ar[rrr]^-{\text{coboundary}}&&&R^2(G\circ F)(X)}$$
    where the column maps are both induced by Grothendieck spectral sequences.
\end{cor}

\begin{proof}[Proof of Lemma \ref{lemma: diagram change 2 to 1}]
    Apply Corollary \ref{coboundspec} for $F=\nu_*$ and $G=\Gamma(X_{\mathrm{Zar}},-)$.
\end{proof}

Finally, we summarize our consequences into the following diagram.

\[\begin{tikzcd}
    & \hol^1(X_{\text{Zar}},\W(\inte_X)/F\left(\W(\inte_X)\right)\otimes_{\inte_X} T_{X/k})\ar{dl}{\gamma}\ar{dr} \\
    \hol^1(X_{\text{Zar}},\nu_*\calZ^{(2)})\ar{rr}\ar{d}&& \hol^1(X_{\text{Zar}},R^1\nu_*T_{X/k}^\sharp)\ar{d}\\
    \hol^1(X_{fpqc},\calZ^{(2)})\ar{rr} && \hol^2(X_{fpqc},T_{X/k}^\sharp)
\end{tikzcd}\]

Let $\widetilde{o}_F\in \hol^1(X_{\text{Zar}},\W(\inte_X)/F\left(\W(\inte_X)\right)\otimes_{\inte_X} T_{X/k})$ be the obstruction of the existence of the splits of $\W(\calO_X)/p\to \calO_X$ and send to an element in $\hol^2(X_{fpqc},T_{X/k}^\sharp)$, which is denoted by $o_{\HT}'$ under the above diagram. Hence, we only need to prove the following theorem.

\begin{theo}\label{theo:compatible oHT}
    We have $o_{\HT}=o_{\HT}'$.
\end{theo}

We will prove the theorem in the next subsection by an explicit calculation of \v Cech cohomologies.

\subsection{Cohomological description of Hodge--Tate gerbes}\label{subsectioncalculation}

The following theorem is taken from \cite[Theorem 5.12]{bhatt2022prismatization}, which is the most important fact we will use to understand $X^{\mathrm{HT}}$.

\begin{theo}\label{definition HT gerbe}
    The stack $X^{\mathrm{HT}}$ is canonically a gerbe banded by $T^{\sharp}_{X/k}:=T_{X/k}\otimes_{\inte_{X}}\mathds{G}_a^\sharp$ (c.f.\cite[Theorem 5.12]{bhatt2022prismatization}) such that the following holds:

    If we are given more that a lifting $\fX/W$ together with a Frobenius lifting $\phi$, there is a canonical isomorphism between $T_{X/k}^{\sharp}$-gerbes
    $$\Psi_{(\fX,\phi)}:X^{\mathrm{HT}}\cong \mathbf{B}T^\sharp_{X/k}.$$
    Precisely, `canonical' means for any $p$-torsion free formal schemes $(\fX_1,\phi_1)$ and $(\fX_2,\phi_2)$ as well as a $\phi$-isomorphism
    $$f:\fX_1\to \fX_2,$$
    we have the following commutative diagram:
    $$\xymatrix{\mathbf{B}T^\sharp_{X_1/k}\ar[r] &\mathbf{B}T^\sharp_{X_2/k}\\ X^{\mathrm{HT}}_1\ar[r]\ar[u] &X^{\mathrm{HT}}_2\ar[u]}$$
\end{theo}

\begin{rmk}
    Keep the above notations. Since $\Psi$ is an isomorphism of gerbes, for any $\Spec(B)\to X$ in $X_{\mathrm{fpqc}}$ and a point $x\in X^{\mathrm{HT}}(\Spec(B))$, let $x'\in \mathbf{B}T^\sharp_{X/k}(\Spec(B))$ be the corresponding point of $x$ under the isomorphism, we have the following diagram commutes:
    $$\xymatrix{\underline{Aut}(x)\ar[r]\ar[d] & \underline{Aut}(x')\ar[d]\\ T^\sharp_{X/k}|_{\Spec(B)}\ar[r]^-{id} & T^{\sharp}_{X/k}|_{\Spec(B)}}$$ where the column arrows are the banded morphisms and the top arrow is induced by the isomorphism.
\end{rmk}

In the following, we choose an affine covering $\{U_\alpha=\Spec(\bar{A}_{\alpha})\}$ and for each $\alpha$ choose $(A_\alpha,\phi_\alpha)$ as a Frobenius lifting of $\bar{A}_\alpha$ over $W$. 

Let $A_{\alpha,\beta}$ be the complete localization of $A_\alpha$ on the affinoid open $U_{\alpha,\beta}:=U_\alpha\cap U_{\beta}$. For each two $\alpha,\beta$, choose an isomorphism $I_{\alpha,\beta}:A_{\alpha,\beta}\cong A_{\beta,\alpha}$ (non-necessarily compatible with $\phi$).

By Theorem \ref{definition HT gerbe}, the gerbe $U_{\alpha}^{\mathrm{HT}}$ is isomorphic to $\mathbf{B}T^{\sharp}_{U_\alpha/k}$. Recall $(A_{\alpha,\beta},\varphi_\alpha)$ and $(A_{\beta,\alpha},\varphi_{\beta})$ provides two trivializations of the gerbe $U_{\alpha,\beta}^{\mathrm{HT}}$. We need to describe the isomorphism between these two trivialisations.

Let $\{V_{\alpha}:=U_\alpha^{\perf}\}$, which form an $\mathrm{fpqc}$-covering of $X$ since $X$ is smooth.

Before the calculation, we describe our strategy. Consider the following diagram
\[\begin{tikzcd}
    \check{\hol}^1\left(\{U_\alpha\to X\},\W(\calO_X)/F\left(\W(\calO_X)\right)\otimes_{\calO_X}T_{X/k}\right)\ar{r} \ar{d} & \hol^1\left(X_{\text{Zar}},\W(\calO_X)/F\left(\W(\calO_X)\right)\otimes_{\calO_X}T_{X/k}\right)\ar{d}\\
    \check{\hol}^1\left(\{U_{\alpha}\to X\},\calZ^{(2)}\right)\ar{r}\ar{d} & \hol^1(X_{\text{fpqc}},\calZ^{(2)}) \ar{d}\\
    \check{\hol}^2(\{V_{\alpha}\to X\}, T^\sharp_{X/k}) \ar{r} & {\hol}^2(X_{\text{fpqc}}, T^\sharp_{X/k}).
\end{tikzcd}\]
All rows are the \v Cech--derived homomorphisms, which are all isomorphisms (cf. \cite[\href{https://stacks.math.columbia.edu/tag/03F7}{Tag 03F7}]{stacks-project}). One the left hand side, the first isomorphism is induced by $\gamma$ defined in Corollary \ref{cor:explicit gamma}, and the second isomorphism has an explicit description (cf. The proof of \cite[\href{https://stacks.math.columbia.edu/tag/03AR}{Tag 03AR}]{stacks-project}). 

Hence, we can write the elements in the right column into \v Cech--cocycles (considered as elements in the left column).

For any finitely many index $\alpha_1,\alpha_2,\dots ,\alpha_k$, we will use
$$U_{\alpha_1,\alpha_2,\dots, \alpha_k}$$
to denote the intersection $\bigcap_{j=1}^kU_{\alpha_j}$. This is affine since $X$ is separate and we will use
$$\bar{A}_{\alpha_1,\alpha_2,\dots,\alpha_k}$$
to denote its coordinate ring.

Recall the section $X^{\mathrm{HT}}(U_{\alpha})$ parameterizes the diagram (of $k$-schemes)
$$\xymatrix{\Spec(\bar{A}_{\alpha})\ar[r]\ar[dr]&\Spec(\W(\bar{A}_{\alpha})/p)\ar[d]\\ &X},$$
where $\W(\bar{A}_\alpha)/p$ is the usual quotient (which coincides to the derived one because $\W(\bar{A}_{\alpha})$ is $p$-torsion free), the top arrow is the projection to the first ghost coordinate and $\Spec(\bar{A}_{\alpha})\to X$ is the open embedding.

The following fact takes from the original proof to \cite[Theorem 5.12]{bhatt2022prismatization}.

\begin{lemma}\label{explicit}
    For each $\alpha$, the trivialisation $\Psi_{(A_\alpha,\phi_{\alpha})}$ in Theorem \ref{definition HT gerbe} is the one described in Lemma \ref{uniqueiso} which sends the trivial torsor to the morphism
    $$\Spec\left(\W(\bar{A}_{\alpha})/p\right)\to U_{\alpha}\to X$$
    associated to the ring homomorphism
    $$\iota_{\alpha}:\bar{A}_{\alpha}\to \W(\bar{A}_{\alpha})/p$$ induced by the $\delta$-structure associated to $\phi_{\alpha}$.
\end{lemma}

\begin{proof}
    See the proof of \cite[Theorem 5.12]{bhatt2022prismatization}.
\end{proof}

\begin{lemma}\label{dercalculation}
    For any $\alpha$, there is an isomorphism $$X^{\mathrm{HT}}(\bar{U}_{\alpha})\cong\mathrm{Hom}_{\bar{A}_{\alpha}}\left(\Omega_{\bar{A}_{\alpha}/k},\W(\bar{A}_{\alpha})/F\left(\W(\bar{A}_{\alpha})\right)\right)$$ where $\W(\bar{A}_{\alpha})/F\left(\W(\bar{A}_{\alpha})\right)$ is the usual quotient as abelian groups and its $\bar{A}_\alpha$-module structure is provided by
    $$(a,x)\mapsto [a]^{p}x\ \forall a\in \bar{A}_{\alpha}, x\in \W(\bar{A}_{\alpha})/F\left(\W(\bar{A}_{\alpha})\right).$$ The second product is the product of witt vectors. This isomorphism depends on the choice of $\phi_{\alpha}$.
\end{lemma}

\begin{proof}
    This is indeed already included in the proof of \cite[Theorem 5.12]{bhatt2022prismatization}.

    Let us first explain why the formula $$(a,x)\mapsto [a]^px$$ defines a $\bar{A}_\alpha$-module structure. This is equal to prove that for any $a,b\in \bar{A}_\alpha$ and $x\in \W(\bar{A}_\alpha)$,
    $$([a]^p+[b]^p)x-[a+b]^px\in p\W(\bar{A}_\alpha).$$ Assume $[a+b]-[a]-[b]=V(y)$ for some $y\in \W(\bar{A}_\alpha)$, $$[a+b]^p-[a]^p-[b]^b=F\left([a+b]-[a]-[b]\right)=F\left(V(y)\right)=py.$$ Hence $([a]^p+[b]^p)x-[a+b]^px=pxy\in p\W(\bar{A}_\alpha).$

    Next we prove that the ideal $$I=V\left(\W(\bar{A}_\alpha)\right)/p\W(\bar{A}_\alpha)$$ of $\W(\bar{A}_\alpha)/p$ satisfies $I^2=0$. In fact, for any $x,y\in \W(\bar{A}_\alpha)$, $$V(x)V(y)=V\left(F\left(V(x)\right)y\right)=V(pxy)=pV(xy)\in p\W(\bar{A}_\alpha).$$ Thus, the sections of $\W(\bar{A}_\alpha)/p\to \bar{A}_\alpha$ over $k$ are parameterized by
    \begin{equation}\label{eq:kanbudongdeshishabi1}\mathrm{Der}_{k}\left(\bar{A}_\alpha,I\right).\end{equation}

    Finally we claim that $I\cong \W(\bar{A}_{\alpha})/F\left(\W(\bar{A}_{\alpha})\right)$ as $\bar{A}_\alpha$ modules. In fact, the isomorphism is given by
    \begin{equation}\label{eq:kanbudongdeshishabi2}\W(\bar{A}_{\alpha})/F\left(\W(\bar{A}_{\alpha})\right)\to V\left(\W(\bar{A}_\alpha)\right)/p\W(\bar{A}_\alpha): x\mapsto V(x).\end{equation}

    Combine (\ref{eq:kanbudongdeshishabi1}) and (\ref{eq:kanbudongdeshishabi2}) the lemma follows.
\end{proof}

Now we can describe the section of $\Psi_{A_{\alpha},\phi_\alpha}$.

\begin{lemma}\label{isosection}
    The isomorphism
    $$X^{\mathrm{HT}}(U_{\alpha})\cong \mathrm{Hom}_{\bar{A}_{\alpha}}\left(\Omega_{\bar{A}_{\alpha}/k},\W(\bar{A}_{\alpha})/F\left(\W(\bar{A}_{\alpha})\right)\right)\cong \hol^1(U_{\alpha,\mathrm{fpqc}},T^\sharp_{Y/k})=\mathbf{B}T^{\sharp}_{X/k}(U_\alpha),$$
    which is defined by combining Lemma \ref{isom torsor} as well as Lemma \ref{dercalculation}, is equal to $\Psi_{A_{\alpha},\phi_\alpha}(U_{\alpha})^{-1}$.
\end{lemma}

\begin{proof}
    See the proof of \cite[Theorem 5.12]{bhatt2022prismatization}.
\end{proof}

\begin{lemma}\label{lemma: V^{-1}(iota-iota)}
    The isomorphism $\Psi_{(A_{\beta},\varphi_{\beta})}\circ\Psi_{(A_{\alpha},\varphi_{\alpha})}^{-1}:\mathbf{B}T^{\sharp}_{X/k}|_{U_{\alpha,\beta}}\cong \mathbf{B}T^{\sharp}_{X/k}|_{U_{\alpha,\beta}}$ is the isomorphism of $T^{\sharp}_{U_{\alpha,\beta}/k}$-gerbes sending the trivial torsor to the image of $V^{-1}(\iota_{\beta}-\iota_{\alpha})$ ($\iota_\alpha,\iota_{\beta}$ were defined in Lemma \ref{explicit}) in $$\hol^{1}\left((U_{\alpha,\beta})_{\mathrm{fpqc}},T^\sharp_{U_{\alpha,\beta}/k}\right)$$ under the isomorphism in \ref{isom torsor}.
\end{lemma}

\begin{proof}
    This is a direct corollary of Lemma \ref{isosection}.
\end{proof}

\begin{defi}\label{sigmah}
    We define $\sigma^h_{\alpha,\beta}$ to be the element in $\hol^{1}\left(({U}_{\alpha,\beta})_{\mathrm{fpqc}},T^\sharp_{U_{\alpha,\beta}/k}\right)$ defined in the above lemma.
\end{defi}

\begin{lemma}\label{lemma: oF tilde rep by cech}
    The cocycle $V^{-1}(\iota_\beta-\iota_\alpha)$ is equal to $\widetilde{o}_F$ in $\hol^1\left(X,T_{X/k}\otimes_{\calO_X}\W(\calO_X)/F(\W(\calO_X))\right)$.
\end{lemma}

\begin{proof}
    This is direct by the definition of obstruction of splitting torsors.
\end{proof}

\begin{defi}\label{sigma}
    Let $\sigma_{\alpha,\beta}$ be the element corresponding to $\sigma^h_{\alpha,\beta}$ (see Definition \ref{sigmah}) in $\fZ^{(2)}(\bar{A}_{\alpha,\beta})$ under the isomorphism Lemma \ref{iso1}. Moreover, we define $\Sigma_{\alpha,\beta}$ to denote the $T^{\sharp}_{X/k}$-torsor on $U_{\alpha,\beta}$ corresponding to $\sigma_{\alpha,\beta}$.
\end{defi}

Note that $\sigma_{\alpha,\beta}$ is equal to the image of $\{V^{-1}(\iota_\beta-\iota_{\alpha})\}$ in $\check{\hol}^1\left(\{U_{\alpha}\to X\},\calZ^{(2)}\right)$.

Now we are ready to calculate the cohomology class of $X^{\mathrm{HT}}$ in $\hol^2\left((X)_{\mathrm{fpqc}},T^{\sharp}_{X/k}\right)$. We will again fix the \textrm{fpqc}-covering $$\{V_{\alpha}\to X\}$$ which we have introduced before. For each $\alpha$, we choose the trivialization $$\Psi_{(A_{\alpha},\varphi_{\alpha})}|_{V_{\alpha}}:X^{\mathrm{HT}}|_{V_{\alpha}}\cong \mathbf{B}T^\sharp_{X/k}|_{V_{\alpha}}$$ and the section $\iota_{\alpha}$ whose preimage under $\Psi_{(A_{\alpha},\varphi_{\alpha})}|_{V_{\alpha}}$ is isomorphic to the trivial torsor as we mentioned above (Lemma \ref{isosection} or \ref{explicit}).

\begin{lemma}
    There is a natural isomorphism
    $$V_{\alpha}\times_X {U}_{\beta}\cong {U}_{\alpha,\beta}^{\perf}.$$
\end{lemma}

\begin{proof}
    This is a direct corollary of the fact that localization commutes with colimit.
\end{proof}

Recall that to find the cohomological class of $X^{\mathrm{HT}}$, we need to find an isomorphism of the sections $\iota_{\alpha}$ and $\iota_{\beta}$ on
$$V_{\alpha}\times_{X}V_{\beta}.$$
Consider the isomorphism of groupoids
$$\Psi_{(A_{\alpha},\varphi_{\alpha})}({V_{\alpha}\times_XV_{\beta}})^{-1}: X^{\mathrm{HT}}|_{V_{\alpha}\times_XV_{\beta}}\to T^\sharp_{X/k}|_{V_{\alpha}\times_XV_{\beta}}.$$
By Lemma \ref{isosection}, $\iota_\alpha$ corresponds to the trivial torsor (on $V_{\alpha}\times_XV_{\beta}$) and $\iota_{\beta}$ corresponds to the torsor $\Sigma_{\alpha,\beta}|_{V_{\alpha}\times_XV_{\beta}}$. By definition, the torsor $\Sigma_{\alpha,\beta}|_{V_{\alpha}\times_XV_{\beta}}$ is constructed by gluing two trivial torsors on $$\left(V_{\alpha}\times_XV_{\beta}\right)\times_X{U}_{\alpha,\beta}^{\perf}$$
along the cocycle
$$\sigma_{\alpha,\beta}'\in \Gamma\left(\left(V_{\alpha}\times_XV_{\beta}\right)\times_X{U}_{\alpha,\beta}^{\perf}\times_X{U}_{\alpha,\beta}^{\perf},T^\sharp_{X/k}\right)$$
where $\sigma_{\alpha,\beta}'$ is just the pull back of $\sigma_{\alpha,\beta}$ along the projection $$\left(V_{\alpha}\times_XV_{\beta}\right)\times_X{U}_{\alpha,\beta}^{\perf}\times_X{U}_{\alpha,\beta}^{\perf}\to {U}_{\alpha,\beta}^{\perf}\times_X{U}_{\alpha,\beta}^{\perf}.$$
Thus an isomorphism of $\iota_{\alpha}$ to $\iota_{\beta}$ is equal to an element
$$\zeta_{\alpha,\beta}\in \Gamma\left(\left(V_{\alpha}\times_XV_{\beta}\right)\times_X{U}_{\alpha,\beta}^{\perf},T^\sharp_{X/k}\right)$$
such that its differential
$$p_{124}^*(\zeta_{\alpha,\beta})-p_{123}^*(\zeta_{\alpha,\beta})$$
is equal to $\sigma_{\alpha,\beta}'$ where
$$p_{123}:\left(V_{\alpha}\times_XV_{\beta}\right)\times_X{U}_{\alpha,\beta}^{\perf}\times_X{U}_{\alpha,\beta}^{\perf}\to \left(V_{\alpha}\times_XV_{\beta}\right)\times_X{U}_{\alpha,\beta}^{\perf}$$
is projection to the first three factors and
$$p_{124}:\left(V_{\alpha}\times_XV_{\beta}\right)\times_X{U}_{\alpha,\beta}^{\perf}\times_X{U}_{\alpha,\beta}^{\perf}\to \left(V_{\alpha}\times_XV_{\beta}\right)\times_X{U}_{\alpha,\beta}^{\perf}$$
is the projection forgetting the third factor.

\begin{defi}\label{o}
    For every $\alpha,\beta$ and $\gamma$, we consider the element $$o_{\alpha,\beta,\gamma}\in\Gamma\left(V_{\alpha}\times_{X}V_{\beta}\times_XV_{\gamma}\times_{X}{U}^{\perf}_{\alpha,\beta,\gamma},T^\sharp_{X/k}\right)$$
    such that $$o_{\alpha,\beta,\gamma}=-p_{124}^*(\zeta_{\alpha,\beta})-p^*_{234}(\zeta_{\beta,\gamma})+p^*_{134}(\zeta_{\alpha,\gamma})$$
    where $p_{ijk}$ is the projection to the product of the $i$-th, $j$-th and $k$-th factor ($U_{\alpha,\beta,\gamma}$ is a localization of $U_{\alpha,\beta}$).
\end{defi}

\begin{theo}\label{theo: calculate class of HT}
    For any $\alpha,\beta$ and $\gamma$, $o_{\alpha,\beta,\gamma}$ lies in $\Gamma\left(V_{\alpha}\times_XV_{\beta}\times_XV_{\gamma},T^{\sharp}_{X/k}\right)$ (viewed as a subgroup of $\Gamma\left(V_{\alpha}\times_{X}V_{\beta}\times_XV_{\gamma}\times_{X}{U}^{\perf}_{\alpha,\beta,\gamma},T^\sharp_{X/k}\right)$ via the covering $p_{123}$). Moreover, they form a $2$-cocycle in $Z^2(\{V_{\alpha}\to X\},T^{\sharp}_{X/k})$ and its image in $\hol^2\left((X)_{\mathrm{fpqc}},T^\sharp_{X/k}\right)$ is the cohomology class of $X^{\mathrm{HT}}$.
\end{theo}

\begin{proof}
    For the first claim, we only need to prove
    $$q_{1}^*(o_{\alpha,\beta,\gamma})-q_{2}^*(o_{\alpha,\beta,\gamma})\in \Gamma\left(V_{\alpha}\times_XV_{\beta}\times_XV_{\gamma}\times_{X}{U}^{\perf}_{\alpha,\beta,\gamma}\times_{X}{U}^{\perf}_{\alpha,\beta,\gamma},T^{\sharp}_{X/k}\right)$$ is equal to $0$
    , where
    $$q_1:V_{\alpha}\times_{X}V_{\beta}\times_XV_{\gamma}\times_{X}{U}^{\perf}_{\alpha,\beta,\gamma}\times_{X}{U}^{\perf}_{\alpha,\beta,\gamma}\to V_{\alpha}\times_{X}V_{\beta}\times_XV_{\gamma}\times_{X}{U}_{\alpha,\beta,\gamma}^{\perf}$$
    is the projection to the first four factors and
    $$q_2:V_{\alpha}\times_{X}V_{\beta}\times_XV_{\gamma}\times_{X}{U}^{\perf}_{\alpha,\beta,\gamma}\times_{X}{U}^{\perf}_{\alpha,\beta,\gamma}\to V_{\alpha}\times_{X}V_{\beta}\times_XV_{\gamma}\times_{X}{U}_{\alpha,\beta,\gamma}^{\perf}$$
    is the projection forgetting the fourth factor.
    Note that
    $$p_{124}^*(\zeta_{\alpha,\beta})-p_{123}^*(\zeta_{\alpha,\beta})\in \Gamma\left(\left(V_{\alpha}\times_XV_{\beta}\right)\times_X{U}_{\alpha,\beta}^{\perf}\times_X{U}_{\alpha,\beta}^{\perf}\right)$$
    is equal to $\sigma_{\alpha,\beta}'$. Thus by definition of $\zeta_{\alpha,\beta}$,
    $$q^*_2p^*_{124}(\zeta_{\alpha,\beta})-q^*_1p^*_{124}(\zeta_{\alpha,\beta})$$ is the pull back of $\sigma_{\alpha,\beta}$
    along the projection
    $$P:V_{\alpha}\times_{X}V_{\beta}\times_XV_{\gamma}\times_{X}{U}^{\perf}_{\alpha,\beta,\gamma}\times_{X}{U}^{\perf}_{\alpha,\beta,\gamma}\to {U}^{\perf}_{\alpha,\beta,\gamma}\times_{X}{U}^{\perf}_{\alpha,\beta,\gamma}\to {U}^{\perf}_{\alpha,\beta}\times_{X}{U}^{\perf}_{\alpha,\beta}.$$
    Similarly,
    $$q^*_2p^*_{234}(\zeta_{\beta,\gamma})-q^*_1p^*_{234}(\zeta_{\beta,\gamma})$$
    is equal to the pull back of $\sigma_{\beta,\gamma}$ along the projection $P$;
    $$q^*_2p^*_{134}(\zeta_{\alpha,\gamma})-q^*_1p^*_{134}(\zeta_{\alpha,\gamma})$$
    is equal to the pull back of $\sigma_{\beta,\gamma}$ along the projection $P$. The claim follows from the cocycle condition of $\sigma_{\text{-},\text{-}}$.

    The second claim follows from Lemma \ref{cechgerbe}.
\end{proof}

\begin{proof}[Proof of Theorem \ref{theo:compatible oHT}]
    Recall the class $o_{\HT}'$ is equal to the image of $\gamma(\widetilde o_{F})$ under the coboundary map
    \[\hol^1(X_{\mathrm{fpqc}},\calZ^{(2)})\to \hol^2(X_{\mathrm{fpqc}},T^\sharp_{X/k}).\]

    The class $\gamma(\widetilde o_{F})$ is represented by the class
    \[\{\sigma_{\alpha,\beta}\}\]
    as a \v Cech class in $\check\hol^{1}(\{V_\alpha\to X\},\calZ^{(2)})$ (cf. Lemma \ref{lemma: oF tilde rep by cech} and the argument after its proof). And the theorem follows from Theorem \ref{theo: calculate class of HT}.
\end{proof}

\begin{theo}\label{theo: section X^HT to frob liftting}
    Let $X/k$ be a smooth variety. There exists a canonical map from the set of equivalent classes of trivializations of $X^\HT$ to the set of isomorphic classes of Frobenius liftings of $X$ over $\W_2(k)$.
\end{theo}

\begin{proof}
    By the definition of $X^\HT$, the set of equivalent classes of trivializations of $X^\HT$ is equal to the set of splittings of the surjection
    \[\W(\calO_X)/p\to \calO_X.\]
    By Theorem \ref{theo: splitting W2=lifting}, the set of isomorphic classes of Frobenius liftings of $X$ over $\W_2(k)$ is equal to the set of splittings of the surjection
    \[\W_2(\calO_X)/p \to \calO_X.\]
    Hence, composition with the natural projection $\W(\calO_X)\to \W_2(\calO_X)$ provides the needed canonical map.
\end{proof}

\section{The weakly $p$-nilpotent Higgs bundles and weakly $p$-nilpotent Hodge--Tate stacks}

In \cite{Ogus_2007}, if we are given a $\W_2(k)$-smooth lifting $\widetilde{X}$ of $X$ (may not has a Frobenius lifting), there exists a non-abelian Hodge theory of so called $p$-nilpotent Higgs bundles. 
It is nature to ask that how to identify those $p$-nilpotent Higgs bundles with certain Hodge--Tate crystals.

We will consider the $T_{X/k}\otimes\alpha_{a}$-gerbe $X_{p-\wnil}^{\HT}$ defined by the obstruction class $o_{F,\mathrm{fpqc}}$ and identify the category of vector bundles with a full subcategory of Hodge--Tate crystals on $(X/\W(k))_{\Prism}$.

\begin{defi}[\ref{defi: weakly p nil HT}]
    Let $X/k$ be a smooth variety, define the \emph{ weakly $p$-nilpotent Hodge--Tate stack $X^{\HT}_{p-\wnil}$ of $X$} as the push-forward of $X^\HT$ by the canonical surjection
    \[ T^{\sharp}_{X/k}\to T_{X/k}\otimes\alpha_p, \]
    which is a $T_{X/k}\otimes \alpha_p$-gerbe.
\end{defi}

By definition, we can rewrite Theorem \ref{theo:main} as the following.

\begin{theo}\label{theo:obs class of HT pnil}
    The obstruction class of $X^{\HT}_{p-\wnil}$ is equal to $o_{F,\mathrm{fpqc}}(X)$.
\end{theo}

Now we assume that $X'$ is a smooth $\delta$-lifting of $X$ over $\W(k)$. By \cite{wang2024prismaticcrystalssmoothschemes}, $X'$ provides an equivalence between the category $\Vect\left((X/W)_\Prism,\overline\calO\right)$ and the category $\mathrm{Higgs}^{\nil}_X$.

Recall the following definition from \cite{lan2012inversecartiertransformexponential}.

\begin{defi}[$p$-nilpotent Higgs bundles]\label{defi: p nil Higgs}
    Let $(E,\theta)$ be a Higgs bundle on $X$. We call $(\calE,\theta)$ $p$-nilpotent if for any open subset $U\subset X$ and any $p$ sections
    \[\partial_1,\partial_2,\dots,\partial_p\in T_{X/k}(U),\]
    the action of $\partial_1\circ\partial_2\circ\dots\circ\partial_p$ on $\calE(U)$ is zero.
\end{defi}

However, our stack $X^{\HT}_{p-\wnil}$ classifies a class of Higgs bundles that satisfy a weaker condition than the $p$-nilpotency. By Theorem \ref{theo:obs class of HT pnil}, the existence of a Frobenius lifting of $X$ coincides with the triviality of $X^{\HT}_{p-\wnil}$. This explains why the Ogus--Vologodsky correspondence exists even when $X$ does not have a Frobenius lifting over $\W_2(k)$, as the class of $p$-nilpotent Higgs bundles is not big enough.

\begin{defi}[weakly $p$-nilpotent Higgs bundles]\label{defi: weak p nil}
    Let $(E,\theta)$ be a Higgs bundle on $X$. We call $(\calE,\theta)$ weakly $p$-nilpotent if for any open subset $U\subset X$ and any section $\partial\in T_{X/k}(U)$,
    the action of $\partial^p$ on $\calE(U)$ is zero.
\end{defi}

The vector bundles on $X^\HT_{p-\wnil}$ can be described by the following theorem.

\begin{theo}
    The pull-back along the natural morphism $X^\HT\to X^\HT_{p-\wnil}$ defines a fully faithful functor from the category $\Vect(X_{p-\wnil}^\HT)$ to the category $\Vect\left((X/W)_\Prism,\overline\calO\right)$. The essential image is identified with the full subcategory of weakly $p$-nilpotent Higgs bundles.
\end{theo}

\begin{proof}
    The $\delta$-lifting $X'$ provides a trivialization of gerbes $\mathbf{B}T^\sharp_{X/k}\cong X^{\HT}$. The equivalence between nilpotent Higgs bundles and Hodge--Tate crystals are given by pull-back of vector bundles. Hence, it is equal to prove that pull-back along the surjection
    \[T^{\sharp}_{X/k}\to T_{X/k}\otimes\alpha_p\]
    provides a fully faithful functor and identifies the essential image with the category of $p$-nilpotent Higgs bundles.

    Now the question is locally on $X$, and we can assume that $T_{X/k}$ is free. Hence, the claim follows from the following lemma.
\end{proof}

\begin{lemma}
    Let $R$ be a ring of characteristic $p$. Then a representation of $\alpha_p$ whose underlying module is locally free of finite rank is identified with a module $M/R$ and an endomorphism $\theta$ of $M$ such that $\theta^p=0$.
\end{lemma}

\begin{proof}
    The Hopf algebra of $\alpha_p$ is equal to $R[T]/T^p$ and the comultiplication map sends $T$ to $T\otimes 1+1\otimes T$.

    Let $M$ be a Hopf module. Then the coaction is a $R$-linear morphism
    \[M\to M\otimes_R R[T]/T^p.\]
    We denote it by $\psi_0+\psi_1T+\dots +\psi_{p-1}T^{p-1}$. We have $\psi_0=id$, as the unit acts trivially on $M$. Also, we have
    \[\sum_{m,n=0}^{p-1}\psi_{m}\circ\psi_nT^m\otimes T^n=\sum_{j=0}^{p-1}\psi_{j}(T\otimes 1+1\otimes T)^j\]
    as morphisms from $M$ to $M\otimes_{R}(R[T]/T^p)\otimes_R(R[T]/T^p)$. This tells us that
    \[\psi_{m}\circ\psi_n=\psi_{m+n}\]
    where $\psi_{k}:=0$ if $k\geq p$. Thus, we have $\psi_{n}=\psi_{1}^n$ and $\psi_1^p=0$. We define our $\theta$ as $\psi_1$.
\end{proof}

\section*{Appendix: Comparing Bhatt--Lurie's construction and Wang's construction}

Throughout the appendix, we fix a perfect field $k$ of characteristic $p$ and a smooth variety $X/k$.

All results in this subsection can be generalized to general prisms $(A,I)$ and smooth schemes $X/\overline{A}$.

Let $(X'/\W(k),\phi')$ be a smooth Frobenius lifting of $X$ over $\W(k)$. Define $\Prism^{(1)}$ as the prismatic closure $\Prism_{X/X'\times X'}$. As both $(X',\phi')$ and $\Prism^{(1)}$ be objects in $\left(X/\W(k)\right)_{\Prism}$, there exists a diagram
\[\begin{tikzcd}
    X\ar{r}\arrow[dr,Rightarrow,shorten=5mm] & X^{\HT}\\
    \overline\Prism^{(1)}\ar{u}\ar{r} & X\ar{u}
\end{tikzcd}\]
The $2$-morphism corresponds to an element in $T^{\sharp}(\overline\Prism^{(1)})$ as $X^\HT$ is a gerbe. By the proof of \cite[Theorem 5.12]{bhatt2022prismatization}, this element can be described as the following.

For simplicity, let $X=\Spec(R)$ be affine, and let $X'=\Spf(R')$, view $\phi'$ as an endomorphism of $R'$. The embedding
\[\iota_0\text{ (resp. $\iota_1$)}:R'\to R'\widehat\otimes_{\W(k)}R',x\mapsto 1\otimes x\text{ (resp. $x\otimes 1$)}\]
induces a $\delta$-homomorphism $\iota_0:R'\to \Prism^{(1)}$ (resp. $\iota_1:R'\to \Prism^{(1)}$). For any $R$-algebra $S$ and $f\in\mathrm{Hom}_{R}(\overline\Prism^{(1)}, S)$, the adjoint property of Witt rings induces a $\delta$-homomorphism
\[\widetilde{f}:\Prism^{(1)}\to \W(S)\]
and $\widetilde{f}\circ\iota_{0}=\widetilde{f}\circ\iota_{1}$. Define
\[\widetilde\psi_f:R'\to \W(S)\]
as sending each $x\in R'$ to $f\left(\frac{\iota_0(x)-\iota_1(x)}{p}\right)$. Then $p\psi_f=0$ and hence $\widetilde\psi_f$ factors through $\psi_f:R\to \W(S)[p]$. Let $\Psi_f=V^{-1}\circ\psi_F$ which is a map from $R$ to $\ker(F:\W(S)\to\W(S))$. The proof of \cite[Theorem 5.12]{bhatt2022prismatization} shows that $\Psi_f$ a derivation.

By the calculation in \cite[Lemma 3.4.11]{bhatt2022absolute}, this is equal to the identification defined by \cite[Proposition 3.18]{wang2024prismaticcrystalssmoothschemes}. Hence, the Bhatt--Lurie's gerbe structure of $X^{\HT}$ is equal to the quotient of the groupoid $\overline\Prism^{(1)}\rightrightarrows X$ in this case. As a corollary, we have the following identification between Wang's construction and Bhatt--Lurie's construction.

\begin{theo}
    The equivalences between the category of nilpotent Higgs bundles and the category of Hodge--Tate crystals defined by Wang (\cite{wang2024prismaticcrystalssmoothschemes}) and Bhatt--Lurie (\cite[Corollary 6.6]{bhatt2022prismatization}).
\end{theo}

\bibliographystyle{alpha}
\bibliography{ref}

\end{document}